\documentclass[11pt]{amsart} 

\usepackage{amsmath,amsthm}
\usepackage{amsfonts}
\usepackage{amssymb}
\usepackage[english]{babel}
\usepackage{xcolor}

\newtheorem{theorem}{Theorem}[section]
\newtheorem{corollary}[theorem]{Corollary}

\newtheorem{proposition}[theorem]{Proposition}
\theoremstyle{definition}

\newtheorem{remark}[theorem]{Remark}

\renewcommand{\P}{{\mathbb P}}
\newcommand{\F}{{\mathcal F}}

\newcommand{\M}{{\mathcal M}}
\newcommand{\B}{{\bf B}}
\newcommand{\E}{{\mathbb E}}
\newcommand{\TT}{{\mathbb T}}

\newcommand{\C}{{\mathbb C}}

\newcommand{\N}{{\mathbb N}}
\newcommand{\Z}{{\mathbb Z}}

\newcommand{\T}{{\mathbb T}}

\theoremstyle{remark}

\begin{document}

\baselineskip=16pt

\title[Resolvent conditions and growth of powers on $L^p$ spaces]
{Resolvent conditions and growth of powers of operators on $L^p$ spaces}

\author{Christophe Cuny}
\address{UMR CNRS 6205, Laboratoire de Math\'ematiques de Bretagne Atlantique, Univ Brest }
\email{christophe.cuny@univ-brest.fr}

\subjclass[2010]{Primary: 47A35, 42A61}
\keywords{Kreiss resolvent condition, power-boundedness, mean ergodicity, Ces\`aro boundedness, Fourier multipliers, type and cotype, UMD spaces}

\begin{abstract}
Let $T$ be a bounded linear operator on $L^p$. We study the rate of growth of the norms of
the powers of $T$ under resolvent conditions or Ces\`aro boundedness assumptions. Actually the relevant properties of $L^p$ spaces  in our study are 
their type and cotype, and for $1<p<\infty$, the fact that they are UMD.
 Some of the proofs make use of Fourier multipliers on Banach spaces,
 which explains why UMD spaces come into play.
\end{abstract}

\maketitle
\vspace*{-0.6 cm}

\section{Introduction}
\smallskip

We study the rate of growth of $\|T^n\|$ for a bounded operator $T$ on a Banach space $X$ under various conditions, continuing recent works of Berm\'udez, Bonilla, M\"uller and Peris \cite{BBMP}, Bonilla and M\"uller \cite{BM} and Cohen, Cuny, Eisner and Lin \cite{CCEL}. In particular, we extend several results of \cite{CCEL}, obtained when $X=H$ is a Hilbert space, to $L^p$ spaces and more generally
to spaces with non trivial type and/or finite cotype.

\medskip

Let us recall the conditions that are relevant to our study. We refer to \cite{BBMP} and 
\cite{CCEL} for more information as well as historical background concerning those conditions. 

\medskip

Let $T$ be a bounded operator on a Banach space $X$. For simplicity we shall assume 
that $X$ is a \emph{complex} Banach space, while all the meaningfull statements 
(i.e. the statements that do not require a complex Banach space in their formulation) hold true
also for real spaces.

\medskip

We say that $T$ is {\it Kreiss bounded} if there exists $C>0$ such that, with 
$R(\lambda,T):=(\lambda I-T)^{-1}$,
\begin{equation} \label{KRC}
\|R(\lambda,T)\| \le \frac C{|\lambda| -1} \qquad \forall |\lambda| > 1\, .
\end{equation}

\smallskip

%We say that $T$ is {\it strongly Kreiss bounded} if there exists $C>0$ such that
%\begin{equation} \label{SKR}
%\|R^k(\lambda,T)\| \le \frac C{(|\lambda| -1)^k} \qquad \forall |\lambda| > 1, \quad k=1,2, \dots
%\end{equation}
% Using the Hille-Yosida theorem, Nevanlinna \cite[Proposition 1.1]{Ne2} 
%proved that $T$ satisfies  (\ref{SKR}) if and only if for some $M$ we have
%\begin{equation} \label{SKR2}
%\|\e^{zT}\| \le M \e^{|z|} \qquad \forall z \in \mathbb C.
%\end{equation}

%\smallskip

We say that $T$ is {\it uniformly Kreiss bounded} if there exists $C>0$ such that
\begin{equation} \label{UKR}
\sup_{n\ge 1} \|\sum_{k=0}^n \frac{T^k}{\lambda^{k+1}}\| \le \frac C{|\lambda|-1} \qquad
\forall |\lambda| > 1. 
\end{equation}

We say that $T$ is {\it absolutely Ces\`aro bounded} if  there exists $C>0$ such that 
\begin{equation} \label{ACB}
\sup_{n\ge 1}\frac1n \sum_{k=0}^{n-1} \|T^k x\| \le C\|x\| \qquad \forall \  x\in X.
\end{equation}

We say that $T$ {\it strongly Ces\`aro bounded} if there is a $C>0$ such that
\begin{equation}\label{SCB}
\sup_{n\ge 1} \frac1n \sum_{k=0}^{n-1} |\langle x^*,T^kx \rangle | \le C\|x^*\|\cdot \|x\|
\quad \forall  (x,x^*) \in X \times  X^*.
\end{equation}

It was proved in \cite{CCEL} that \eqref{SCB} is equivalent to the existence of $C>0$ such that
\begin{equation}\label{SCB2}
\sup_{n\ge 1} \sup_{|\gamma_0|=1,\ldots , |\gamma_{n-1}|=1} 
\frac1n \| \sum_{k=0}^{n-1} \gamma_k T^kx \| \le C \|x\|
\quad \forall x\in X\, .
\end{equation}

Let us mention the following implications concerning those conditions. First of all any power bounded operator $T$, i.e. such that $\sup_{n\ge 0}\|T^n\|<\infty$, satifies all of the above conditions. If $T$ is uniformly Kreiss bounded it is also 
Kreiss bounded; if $T$ is absolutely Ces\`aro bounded it is strongly Ces\`aro bounded, 
hence \cite{CCEL} uniformly Kreiss bounded.
The converse of the above implications do not hold in general. 
%Finally, the notion of strong Kreiss boundedness and absolute Ces\`aro boundedness are independent.

\medskip

%Let $T$ be strongly Kreiss bounded on a Banach space $X$. Then, see e.g. \cite{LN}, 
%$\|T^n\|=O(\sqrt n)$ and this estimate is best possible in general Banach spaces. 
%However, when $X$ is a Hilbert space, by \cite[Theorem 4.5]{CCEL}, $\|T^n\|=O((\log n)^\kappa)$ for 
%some $\kappa>0$, and, by Proposition 4.9 there, this estimate is best possible in Hilbert spaces. 

%\medskip

Let $T$ be absolutely Ces\`aro bounded on a Banach space $X$. Then, see \cite[Proposition 3.1]{CCEL}, $\|T^n\|=O(n^{1-\varepsilon})$ for some $\varepsilon \in (0,1)$ and this estimate is best possible in general Banach spaces. Earlier, the estimate $\|T^n\|=o(n)$ was proved in \cite{BBMP}. If $X$ is a Hilbert space then, by Theorem 4.4 of \cite{CCEL}, 
$\|T^n\|=O(n^{1/2-\varepsilon})$ for some $\varepsilon\in (0,1/2)$ and this is best possible. 

\medskip

Let $T$ be Kreiss bounded on a Banach space. Then, by \cite{LN}, $\|T^n\|=O(n)$ and this is best possible in general Banach spaces by an example of 
Shields \cite{Sh}. If $X$ is a Hilbert space then, see \cite[Theorem 4.1]{CCEL} or \cite{BM}, $\|T^n\|=O(n/\sqrt {\log n})$. We do not know whether this is optimal. As far as we know the only result in that direction is that for every $\varepsilon\in (0,1)$, there exists $T$ on a Hilbert space that is uniformly Kreiss bounded and such that $\|T^n\|=O(n^{1-\varepsilon})$.
 This is proved in \cite{BM}, where in fact $T$ is even strongly Ces\`aro bounded \cite{CCEL}.

\medskip

In this paper we obtain estimates of $\|T^n\|$ for absolutely Ces\`aro bounded,  strongly
Ces\`aro bounded or Kreiss bounded operators, according to the type and/or cotype of $X$. 
Some results only hold on UMD spaces, see later for the definitions. 
Estimates when $X=L^p(\Omega,\mu)$, $1<p<\infty$, are obtained as corollaries.

%Van Casteren \cite{VC2} proved that if {\it both} $T$ and $T^*$ are Ces\`aro square 
%bounded in $H$, then $T$ is power-bounded, and gave an example in $\ell^2(\mathbb Z)$ 
%of $T$ not power-bounded satisfying (\ref{CSB}). Zwart \cite{Zw} gave a simpler proof 
%of power-boundedness, in any Banach space, when $T$ and $T^*$ both satisfy (\ref{CSB}).
%In (a) $\Leftrightarrow$ (d) of \cite[Theorem 2.3]{CS}, 
%Chen and Shaw extended Zwart's result.
%(obtained with $\alpha=0$ and $p=2$);
%however, since for {\it positive} sequences Ces\`aro boundedness and Abel boundedness are equivalent 
%(e.g. \cite[1.5-1.7]{Em}), the use of "Abel square boundedness"  in \cite{CS} is not more general.
%that $T$ is power-bounded if (and only if) both $T$ and $T^*$ are Abel square bounded, 
%thus extending Zwart's result.  

\medskip

\bigskip

\section{Growth of the powers for absolutey Ces\`aro bounded operators}

%\section{Absolutely Ces\`aro bounded operators on $L^p$ spaces}

In this section we study the growth rate of $\|T^n\|$ when $T$ is an absolutely Ces\`aro bounded operator on a Banach space of type $p$ and 
cotype $q$, and then apply the results to $L^p$ spaces. We recall the definitions \cite[p. 151]{AK}.

{\bf Definition.} A Banach space $X$ is said to be {\it of type $p\in[1,2]$} if there exists $K>0$ 
such that for any $n\in \N$ and any $x_1,\ldots , x_n\in X$, one has 
$$
\E(\|\varepsilon_1 x_1+\ldots +\varepsilon_nx_n\|^p)\le K( \|x\|^p+\ldots +\|x_n\|^p\,)\, , 
$$
where  $(\varepsilon_1,\ldots , \varepsilon_n)$ are iid Rademacher random variables (defined on $[0,1]$
with Lebesgue's measure $\lambda$; the expectation $\E$ is integration, see \cite[p. 145-6]{AK}). 

A Banach space $X$ is said to be of finite {\it  cotype $q\ge 2$} if there exists $K>0$ such 
that for any $n\in \N$  and any $x_1,\ldots , x_n\in X$, one has 
$$
\|x_1\|^q+\ldots +\|x_n\|^q\le K\E(\|\varepsilon_1 x_1+\ldots +\varepsilon_nx_n\|^q)\,. 
$$

A Banach space is said to be of \emph{cotype} $\infty$ if 
there exists $K>0$ such 
that for any $n\in \N$  and $x_1,\ldots , x_n\in X$, one has 
\begin{equation}\label{cotype-infini}
\max_{1\le i\le n}\|x_i\|\le K\E(\|\varepsilon_1 x_1+\ldots +\varepsilon_nx_n\|)\,. 
\end{equation}

\smallskip

\medskip

Every Banach space is of type 1 and (using for instance $(ii)$ of Proposition 
\ref{contraction-principle} below) of cotype $\infty$. A Banach space with type $1<p\le 2$ 
is said to have non-trivial type, and a Banach space with cotype $q\in [2,\infty)$ is said to 
have finite cotype.  If $X$ has non trivial type, it has finite cotype (see Theorem 7.3.11 page 98 of \cite{HNVW2}) but the converse is wrong (any $L^1$ space has cotype 2 but has trivial type, see page 154 of \cite{HNVW2}). If $X$ is of 
type $p> 1$ then $X^*$ is of (finite) cotype $q=p/(p-1)$ (see Theorem 7.1.13 page 63  \cite{HNVW2}). If $X$ is of finite cotype $q$, $X^*$ may be of trivial type (take $X=L^1$ again) but $X^*$ is of non trivial type $p=q/(q-1)$ if we further assume that $X$ has non trivial type (see Theorem 7.4.10 page 114 of \cite{HNVW2} and recall that $X$ is $K$-convex if and only if it has non trivial type). Finally, let us mention that $X$ has non trivial type if and only if $X^*$ does (see page 124 of \cite{HNVW2}).

\medskip

Typical examples of Banach spaces with non trivial type and finite cotype are given by the reflexive $L^p$-spaces. Indeed (see page 154 of \cite{AK}) when $X=L^p(\nu)$ for some $\sigma$-finite measure $\nu$, $X$ has type 
$p'=\min(p,2)$ and cotype $p''=\max(p,2)$ 
(and this is is best possible if the space is not finite dimensional).
More generally (combine Theorem 10.1 with Propositions 10.1 and 10.2 of \cite{Pisier}) uniformly convexifiable Banach spaces have non trivial type and finite cotype. Those spaces are again relexive. 
\medskip

Now, for a general Banach space,  there is no relation between the property of being reflexive and the property of having non trivial type and finite cotype. Taking $1< p_n<\infty$ with $\lim_n p_n=+\infty$  and for $X$ the $\ell^2$ direct sum of 
$\ell^{p_n}$ one obtains a reflexive Banach space with trivial type and only finite cotype ($\ell^{p_n}$ has type 2 but the best constant in the definition must go to $\infty$ as $p_n\to \infty$ since $\ell^\infty$ has trivial type). Notice that this example is such that $X^*$ has cotype 2).  Moreover, for every $\varepsilon>0$ there 
exists a non reflexive Banach space with type 2 and cotype 
$2+\varepsilon$ by Corollary 12.20 page 492 of \cite{Pisier}. However, one cannot take $\varepsilon=0$ since by a result of Kwapie\'n (see Theorem 7.3.1 page 89 of \cite{HNVW2}) every Banach space with type 2 and cotype 2 is isomorphic to a Hilbert space.

\medskip

We shall need a somewhat direct consequence of the definition of type and cotype. By Proposition 9.11 of \cite{LT}, if $X$ is of finite cotype $q$ then, for every independent integrable and centered (i.e. $\E(\xi_i)=0$) $X$-valued variables $\xi_1,\ldots \xi_n$, we have 
\begin{equation}\label{LT-cotype}
\E(\|\xi_1+\ldots +\xi_n\|^q)\ge 
C_q\E(\|\xi_1\|^q)+\ldots +\E(\|\xi_n\|^q)\, .
\end{equation}
When $X$ has type $p$ we have a reverse inequality 
\begin{equation}\label{LT-type}
\E(\|\xi_1+\ldots +\xi_n\|^p)\le 
C_p\E(\|\xi_1\|^p)+\ldots +\E(\|\xi_n\|^p)\, .
\end{equation}

\medskip

 We will need Kahane-Khintchine's inequalities \cite[p. 148]{AK},  
which we recall for convenience in the following form.

\begin{theorem}
Let $X$ be a Banach space. For every $p,q>0$ there exists $C_{p,q}>0$ such that for every 
$x_1,\ldots , x_n\in X$, 
$$
(\E(\|\varepsilon_1x_1+\ldots +\varepsilon_n x_n\|^p)^{1/p}\le 
C_{p,q}(\E(\|\varepsilon_1x_1+\ldots +\varepsilon_n x_n\|^q)^{1/q} \, .
$$
\end{theorem}

We will need the following corollaries of Kahane's contraction principle 
\cite[Theorem 6.1.13(ii)]{HNVW2}. The first follows by application to $b_k\xi_k$ with
$a_k=1/b_k$, and the second  follows from \cite{HNVW2} by taking $a_k=1$ for $k \in I$ 
and $0$ for $k \notin  I$.

\begin{proposition} \label{contraction-principle}
Let $1 \le p \le \infty$ and let $(\xi_k)_{k=1}^n$ be independent $\mathbb R$-symmetric
$X$-valued random variables in $L^p(\Omega,\lambda;X)$. 
Then for $I \subset J \subset \{1,2,\dots,n\}$ we have:
\smallskip

(i) $\displaystyle{\E(\|\sum_{k\in I}  \xi_k\|^p)  \le 
\Big(\frac\pi 2\Big)^p \max_{j \in I}\Big|\frac1{b_{j}}\Big|^p \E(\| \sum_{k\in I} b_k \xi_k\|^p)}$ 
when $b_j \ne 0$ for any $j$. 
\smallskip

(ii) $\E(\|\sum_{k\in I} \xi_k\|^p) \le  \E(\|\sum_{k\in J} \xi_k\|^p)$.
\end{proposition}
\medskip

We are now in position to state and prove the main result of this section.

\begin{theorem}
Let $T$ be an absolutely Ces\`aro bounded operator on a Banach space $X$ of type $1\le p\le 2$.
Then $\|T^n\|= O(n^{1/p})$. Moreover, there exists  $C>0$ such that for every $x\in X$, every 
sequence of iid Rademacher variables $(\varepsilon_n)_{n\ge 0}$ and every $n\in \N$, 
\begin{equation}\label{pq-bounded1}
\E(\|\sum_{k=0}^{n-1}\varepsilon_kT^kx\|^p)\le C n\|x\|^p\, .
\end{equation}

If in addition $X$ is of finite cotype $q\ge 2$, then 
$\|T^n\|=O(n^{1/p}/(\log n)^{1/q})$, and there exists $\tilde C>0$ such that for every $x\in X$ 
and every $n\in \N$, 
\begin{equation}\label{pq-bounded}
\sum_{k=0}^{n-1}\|T^kx\|^q\le \tilde C n^{q/p}\|x\|^q\, .
\end{equation}
\end{theorem}

\begin{remark}
As mentionned previously, when $p>1$ then $X$ automatically has finite cotype and the second part  of the theorem applies. Both items of the theorem apply to uniformly convexifiable Banach spaces.
\end{remark}

\begin{proof}
The bound  $\|T^n\|= O(n^{1/p})$ will follow from (\ref{pq-bounded1}) and item $(ii)$ of Proposition \ref{contraction-principle}.
%$$
%\|T^nx\|= \E(\|\varepsilon_nT^nx\|)\le 
%\E(\|\sum_{k=0}^n \varepsilon_kT^k x\|) + \E(\|\sum_{k=0}^{n-1} \varepsilon_kT^k x\|) \le
%$$
%$$
%\big(\E(\|\sum_{k=0}^n \varepsilon_kT^k x\|^p)\big)^{1/p} + 
%\big(\E(\|\sum_{k=0}^{n-1} \varepsilon_kT^k x\|^p)\big)^{1/p}  \le 2C^{1/p}(n+1)^{1/p} \|x\|\, .
%$$
\medskip

Let us prove \eqref{pq-bounded1}.
By the contraction principle, it suffices to prove that, for every $N\ge 0$, 
$$
\E(\|\sum_{k=0}^{2^N-1}\varepsilon_kT^kx\|^p)\le C 2^N\|x\|^p\, .
$$
Using that $X$ has type $p$, by \eqref{LT-type}, it suffices to prove that, for every $N\ge 1$, 
\begin{equation} \label{blocks}
\E(\|\sum_{k=2^{N-1} }^{2^N-1}\varepsilon_kT^kx\|^p)\le C 2^N\|x\|^p\, .
\end{equation}

Let $N\in \N$. Fix $x \in X$ with $\|x\|=1$. Let $(\varepsilon_n)_{n\in \N}$ be the 
Rademacher system on $([0,1],\lambda)$. Denote by $S$ the Koopman operator 
associated with multiplication by 2 (mod 1), so that $S\varepsilon_n=\varepsilon_{n+1}$.

Define
$$
y_{N}:=  \sum_{k=1}^{2^N}\frac{\varepsilon_k T^kx}{\|T^kx\|+1}
$$
and 
%\begin{align*}
$$
u_{N}    := \sum_{j=0}^{2^N-1} (S^j \otimes T ^j) y_{N} =
\sum_{j=1}^{2^N}\varepsilon_j T^j x\sum_{k=1}^j \frac1{\|T^kx\|+1} +
\sum_{j=2^N+1}^{{2^{N+1}}}\varepsilon_j T^j x\sum_{k=j-2^N}^{2^N} \frac1{\|T^kx\|+1}\, . 
$$
% \end{align*}

Since $T$ is absolutely Ces\`aro bounded, by the computations in (17) of \cite{CCEL} with $\varepsilon=1$ and $N$ instead of $2^{N-1}$ (the computations in \cite{CCEL}  are done in the Hilbert case but just make use of the norm, hence apply equally in general Banach spaces), we obtain
\begin{equation} \label{reciprocals2}
\sum_{k=1}^N\frac1{\|T^k x\|+1}\ge \tilde C N \, , 
\end{equation}
for some $\tilde C>0$ (independent of  $N$ and of $x$ with norm 1). 

We now use the contraction principle twice, first item $(ii)$, then item $(i)$, and we use  \eqref{reciprocals2} for the last inequality, to obtain
$$
\frac\pi 2 \|u_N\|_{L^p(X)} \ge \frac\pi 2 
\Big\| \sum_{j=2^{N-1}+1}^{2^N}\varepsilon_jT^jx(\sum_{k=1}^j \frac1{\|T^kx\|+1}) \Big\|_{L^p(X)} \ge
$$
$$
\min_{2^{N-1}+1\le \ell\le 2^N}\Big(\sum_{k=1}^\ell \frac1{\|T^kx\|+1}\Big) 
\ \Big\| \sum_{j=2^{N-1}+1}^{2^N}\varepsilon_jT^jx \Big\|_{L^p(X)} =
$$
$$
\Big(\sum_{k=1}^{2^{N-1}+1} \frac1{\|T^kx\|+1}\Big)\ \Big\| \sum_{j=2^{N-1}+1}^{2^N}\varepsilon_jT^jx \Big\|_{L^p(X)}\ge
\frac{\tilde C}2 2^N \Big\| \sum_{j=2^{N-1}+1}^{2^N}\varepsilon_jT^jx \Big\|_{L^p(X)} \ .
$$
Hence
\begin{equation}\label{minor}
\Big\| \sum_{j=2^{N-1}+1}^{2^N}\varepsilon_jT^jx \Big\|_{L^p(X)}^p \le C' 2^{-Np} \|u_N\|_{L^p(X)}^p\, .
\end{equation}

Since $X$ is of type $p$, 
$$
\E(\|y_N\|^p)\le  K_p \sum_{k=1}^{2^N} \Big(\frac{\|T^kx\|}{\|T^kx\|+1}\Big)^p \le K_p 2^{N}\, . 
$$
%By Kahane's inequalities, 
Hence $\E(\|y_N\|)\le \E(\|y_N\|^p)^{1/p} \le K_p^{1/p} 2^{N/p}$.
By stationarity of the Rademacher system, $\E (\|(S^j\otimes T^j)y_N\|)= \E(\| T^j y_N\|)$
 and by the  absolute Ces\`aro boundedness of $T$, we obtain
\begin{equation} \label{uN-L1-norm}
\E(\|u_N\|)\le \sum_{j=0}^{2^N-1}\E(\| (S^j \otimes T ^j)y_{N}\|\le 
C \E(2^N\|y_N\|) \le D 2^{N+N/p}\, .
\end{equation}
Applying Kahane-Khintchine's inequality to the decomposition of $u_N$, we obtain
$$
\|u_N\|_{L^p(X)} =(E(\|u_N\|^p)^{1/p} \le C_{p,1}\E(\|u_N\|) \le C_{p,1}D2^{N+N/p}\ .
$$
Combining the last estimate with \eqref{minor} we obtain
$$
\E\Big( \Big\| \sum_{j=2^{N-1}+1}^{2^N}\varepsilon_jT^jx \Big\|^p\Big) \le \tilde K 2^{-Np}\E(\|u_N\|^p)
\le C2^N \ .
$$

\medskip

We now assume that $X$ has  finite cotype $q$. 
%and denote by $r=q/(q-1)$ the dual index.

Using the definition of cotype, Kahane-Khintchine's inequalities and  (\ref{pq-bounded1}), we obtain
$$
\sum_{k=0}^{n-1} \|T^kx\|^q \le K\E(\| \sum_{k=0}^{n-1}\varepsilon_k T^kx \|^q) \le
 KC_{q,p}^q \E(\| \sum_{k=0}^{n-1}\varepsilon_k T^kx \|^p)^{q/p} \le
\tilde C n^{q/p}\|x\|^q\, ,
$$ 
which proves \eqref{pq-bounded}.  However, this yields only $\|T^n\|= O(n^{1/p})$.
\smallskip

Denote by $r=q/(q-1)$ the dual index of $q$.
Since $T$ is absolutely Ces\`aro bounded, it is strongly Ces\`aro bounded, and by
\cite[Corollary 3.7]{CCEL}  so is $T^*$. By Kahane's inequalities and \cite[Proposition 3.6]{CCEL}, 
for $x^* \in X^*$  and $Q>P \ge 0$ we have
$$
\Big( \E\big(\|\sum_{k=P}^{Q-1} \varepsilon_k T^{*k}x^* \|^r \big) \Big)^{1/r} \le
C_{r,1} \E(\|\sum_{k=P}^{Q-1} \varepsilon_k T^{*k}x^* \|) \le 2Q \cdot C_{r,1} K_{scb} \|x^*\|.
$$

For every $(x,x^*)\in X\times X^*$ and integers $N\ge Q>P \ge 0$,  we have 
\begin{gather*}
(Q-P)|\langle x^*, T^Nx \rangle|=
\Big|  \E\Big( \sum_{\ell=P}^{Q-1}\sum_{k=P}^{Q-1}
\langle \varepsilon_kT^{*k}x^*,\varepsilon_\ell T^{N-\ell}x\rangle\Big)  \Big| \le \\ 
  \E\Big( \Big| \langle \sum_{\ell=P}^{Q-1} \varepsilon_kT^{*k}x^*,
\sum_{k=P}^{Q-1} \varepsilon_\ell T^{N-\ell}x \rangle\Big|  \Big) \le \\
\Big(\E \Big(\big\| \sum_{k=P}^{Q-1}\varepsilon_k T^{*k}x^*\big \|^r\Big)\, \Big)^{1/r} \, 
\Big(\E \Big( \big\|\sum_{\ell=P}^{Q-1}\varepsilon_\ell T^{N-\ell}x\big\|^q\Big) \, \Big)^{1/q}\le
\\ CQ \|x^*\| \cdot
\Big(\E \Big( \big\|\sum_{\ell=P}^{Q-1}\varepsilon_\ell T^{N-\ell}x\big\|^q\Big) \, \Big)^{1/q}\, ,
%CQ \Big(\E \Big(\big\| \sum_{k=P}^{Q-1}\varepsilon_k T^{*k}x^*\big \|^q\Big)\, \Big)^{1/q}\, .
\end{gather*}
with $C=2C_{r,1}K_{scb}$. Taking the supremum over $\{\|x^*\|=1\}$ we conclude that 
\begin{equation}\label{claim-BBMP1}
\frac{(Q-P)^q}{Q^q} \|T^Nx\|^q \le 
C^q \E \Big(\big\| \sum_{\ell=P}^{Q-1}\varepsilon_\ell T^{N-\ell}x\big \|^q\Big) \qquad \quad
0\le P<Q \le N\, .
\end{equation}
\smallskip

Fix $N\in \N$ and put $L:=\log (N/2)/\log 2$. It follows from \eqref{claim-BBMP1} 
that for every $0\le \ell \le L$,
\begin{equation} \label{block}
\E\Big(\big \|\sum_{k=N+1-2^{\ell +1}}^{N-2^\ell} \varepsilon_k T^k x\big \|^q\Big)\ge 
2^q \|T^Nx\|^q/C^q\, .
\end{equation}
Denote $Z_\ell: \sum_{k=N+1-2^{\ell +1}}^{N-2^\ell} \varepsilon_k T^k x$. Then $(z_\ell)$
are independent on  $\Omega=([0,1],\lambda)$.

 Since $X$ has cotype $q$, by \eqref{LT-cotype},
$$
\sum_{l=0}^L\E(\|z_l\|^q )\le K  \E(\| \sum_{l=0}^L \varepsilon_\ell z_\ell \|^q) \, .
$$
Using Item $(ii)$ of Proposition \ref{contraction-principle},
$$
 \E(\|\sum_{l=0}^L y_l\|^q)   \le
 \E(\|\sum_{k=0}^{N-1} \varepsilon_k T^k x\|^q)\, .
$$
 Combining with (\ref{block}), we obtain
%Hence, using that $X$ has cotype $q$ and Kahane's inequalities, we infer that
\begin{gather*}
2^q (L+1) \|T^Nx\|^q /C^q \le \sum_{\ell=0}^L \E(\|z_\ell\|^q ) \le
 K \cdot \E(\|\sum_{k=0}^{N-1} \varepsilon_k T^k x\|^q) \, .
%\le C'N^{q/p}\, .
%\big \|\sum_{k=N+1-2^{\ell +1}}^{N-2^\ell}\varepsilon_k T^kx\big \|^q\Big) \ge 
\end{gather*}
 Using Kahane's inequalities and (\ref{pq-bounded1}), we conclude that 
$\|T^N x\| \le C'N^{1/p} \|x\| / L^{1/q}$.
\end{proof}

\begin{remark} 1. When $X$ is of type $p=1$, the Theorem yields for $T$ absolutely Ces\`aro bounded that
$\|T^n\|=O(n/(\log n)^{1/q})$; the estimate $\|T^n\|=O(n^{1-\varepsilon})$ in 
\cite[Proposition 3.1]{CCEL} is better.

2. When $X$ has type $p>1$, the estimate $\|T^n\|=O(n^{1/p})$ for $T$  absolutely Ces\`aro bounded
should be compared with $O(n^{1- 1/K_{ac}})$ (with $K_{ac}$ the best constant in the definition of absolute Ces\`aro boundedness) given in  \cite[Proposition 3.1]{CCEL}
(with $p=1$ there). The estimate of the theorem is better when $K_{ac} > p/(p-1)$.
\end{remark}

\begin{corollary}\label{cor-Lp}
Let $T$ be absolutely Ces\`aro bounded on $L^p({\Omega}, \mu)$, $1\le p<\infty$, with 
$\mu$ $\sigma$-finite. Then, $\|T^n\|_p= O(n^{1/p'}/(\log n)^{1/p''})$ where $p'=\min(p,2)$
and $p''=\max(p,2)$.
Moreover, there exists $C_p>0$ such that 
\begin{equation}\label{Lp-bounded}
\sum_{k=0}^{N-1}\|T^kx\|_p^{p''}\le C_p N^{p''/p'}\|x\|_p^{p''}\, ,
\end{equation}
\end{corollary}
\begin{proof} It is well known \cite[p. 154]{AK} that $L^p({\Omega},\mu)$ is of type $p'$ 
and of cotype $p''$.  
\end{proof}

\begin{remark}
 1. When $1\le p\le 2$, the bound \eqref{Lp-bounded} reads 
$\sum_{k=0}^{N-1}\|T^kx\|_p^2\le C_p N^{2/p}\|x\|_p^2$, which is implied by the bound 
$\sum_{k=0}^{N-1}\|T^kx\|_p^p\le C_p^{p/2} N\|x\|_p^p$. When $p\ge 2$, the bound \eqref{Lp-bounded} 
reads $\sum_{k=0}^{N-1}\|T^kx\|_p^p\le C_p N^{p/2}\|x\|_p^p$, which is implied by the bound 
$\sum_{k=0}^{N-1}\|T^kx\|_p^2\le C_p^{2/p} N\|x\|_p^p$. 

2. For $p=2$, \eqref{Lp-bounded} shows that $T$ is Ces\`aro square bounded, as was also shown in
\cite[Theorem 4.3]{CCEL}.  However, the bound on $\|T^n\|$ in \cite[Theorem 4.3]{CCEL}, obtained 
from \cite[Proposition 3.1]{CCEL},  is better than that of Corollary \ref{cor-Lp}.

3. Comparing the examples of \cite[Theorem 3.3]{CCEL} with Corollary  \ref{cor-Lp}, we see that 
when $p\in [1,2]$, the corollary gives the correct bound for $\|T^n\|$, up to some $\varepsilon>0$ 
in the exponent. One may wonder whether the bound $\|T^n\|=O(n^{1/2})$, provided by the corollary 
when $p>2$, is the right one; in the examples of \cite[Theorem 3.3]{CCEL}, 
$\|T^n\| \le (n+1)^{1/p} <(n+1)^{1/2}$.  We provide below another class of examples, which show 
that for $p>2$, the bound in Corollary \ref{cor-Lp} is close to optimal.
\end{remark}

\medskip

Let $\TT:=\{\gamma\in \C\, :\, |\gamma|=1\}$ and let $\lambda$ be that Haar measure on $\TT$.
\begin{proposition}
Let $2<p<\infty$. For $\varepsilon\in (0,1/2)$ there exists an absolutely Ces\`aro bounded $T$ on 
$L^p(\TT,\lambda)$ such that  $\|T^n \|\asymp n^{1/2-\varepsilon }$
\end{proposition}
\begin{proof} 
%Let $\varepsilon>0$ and $1<p<\infty$. 
We first define a projection $Q\, :\, L^2(\TT,\lambda)\to L^2(\TT,\lambda)$ as follows. 
For $f=\sum_{n\in \Z} c_n \gamma^n\in L^2(\TT,\lambda)$ set $Qf:=\sum_{m\ge 1}c_{2^m}\gamma ^{2^m}$. 
Notice that $Q$ is the operator obtained by multiplying the Fourier coefficients term by term 
with the sequence $(a_n)_{n\in \Z}$ given by $a_{2^m}=1$ for every $m\ge 1$ and $a_n=0$ otherwise. 
Since the sequence $(a_n)_{n\in \Z}$ has bounded dyadic variation,  by the Marcinkiewicz multiplier
theorem \cite[Theorem XV(4.14)]{Zy} it defines a bounded Fourier multiplier on $L^p(\TT,\lambda)$ 
%(see Section $**$ for the definitions). In particular, 
(i.e. $\|\sum_{n\in \Z} a_nc_n \gamma^n\|_p \le A_p \|\sum_{n\in \Z} c_n \gamma^n\|_p$ for 
$\sum_{n\in \Z} c_n \gamma^n$  in $L^p$). Thus $Q$ extends to  a bounded operator on $L^p(\mathbb T)$, 
and  for every $f\in L^p$  we have, by \cite[Theorem V(8.20)]{Zy}
%using the Littlewood-Paley theorem for the left hand inequality, 
\begin{equation} \label{PW}
\frac1{C_p}(\sum_{m\ge 1} |c_{2^m}|^2\big)^{1/2}\le \|Qf\|_p \le C_p (\sum_{m\ge 1} |c_{2^m}|^2\big)^{1/2}\, ,
\end{equation}
where $C_p$ depends only on $p$.

By \eqref{PW}, $Q$ actually takes values in 
$M:=\{g=\sum_{m\ge 1} b_m \gamma^{2^m}\, : \, (b_m)_{m\in \N}\in \ell^2(\N)\}$, which is closed
in $L^p$, and is clearly isomorphic to $\ell^2(\N)$. For $g \in M$, \eqref{PW} yields 
that $\|g\|_p\sim \|g\|_2$.

We now define $R : \, M\to M$ as follows. For $g=\sum_{m\ge 1}b_m\gamma^{2^m}\in M$
set $Rg:= \sum_{m\ge 1}d_m\gamma^{2^m}$, where $d_m= \Big(\frac{m+1}m\Big)^{1/2-\varepsilon} b_{m+1}$.
\smallskip

Finally, we set $T:=RQ$. Since $Q$ is a projection of $L^p$ onto $M$ and $R$ takes values in $M$, 
we see that $T^n=R^n Q$ for every $n\ge 1$, so $T$ is absolutely Ces\`aro bounded on $L^p$
whenever $R$ is (on $M$), and the desired estimate on $\|T^n\|$ follows from the same estimate for $R$. 

\smallskip

But, since $M$ and $\ell^2(\N)$ are isomorphic, we see that  $R$ is similar to the weighted 
backward shift on $\ell^2(\N)$ defined in Theorem 2.1 of \cite{BBMP}, so the estimates of
\cite{BBMP} finish the proof. 
\end{proof}

\medskip

\medskip
In view of the above remarks and the known examples of absolutely 
Ces\`aro bounded operators on 
$L^p$ spaces, and in view of the results in the Hilbert case, the following question seems natural. The notion of $p$-absolute Ces\`aro boundedness was defined in \cite{CCEL}; for $p=2$ see also \cite{BBMP}.

\medskip

\noindent {\bf Question.} Let $T$ be an absolutely Ces\`aro bounded operator on an 
$L^p$-space with $1\le p <2$ (resp. with $p> 2$); is $T$  $\ p$-absolutely Ces\`aro bounded 
(resp. $2$-absolutely Ces\`aro bounded)?

\medskip
\medskip

\section{Growth of the powers for Kreiss bounded operators} \label{UMD}

Montes-Rodr\'\i guez et al. \cite{MSZ} and Aleman-Suciu \cite{AS} asked whether any uniformly Kreiss bounded operator $T$ satisfies $\|T^n\|=o(n)$ (hence is mean ergodic whenever $X$ is reflexive). In this section, we prove that every Kreiss bounded operator $T$ on a  UMD space, 
satisfies $\|T^n\|=o(n)$, with even a logarithmic rate.

\medskip

{\bf Definition.}
We say that a Banach space $X$ is {\it UMD (Unconditional Martingale Differences
property)} if for some (every) $p>1$, there exists $C_p>0$ such that for every sequence 
$(d_n)_{1\le n\le N}$ of martingale differences in some $L^p(\Omega,X,\nu)$ 
%(with underlying $\sigma$-finite measure space $(S,\SS,\nu)$) 
and every sequence $(\varepsilon_n)_{1\le n\le N}\in \{-1,1\}^N$, we have
$$
\|\sum_{n=1}^N \varepsilon_n d_n\|_{L^p(\Omega,X)}\le C_p \|\sum_{n=1}^Nd_n\|_{L^p(\Omega,X)}\, .
$$

%We shall see that the previous argument carry on with little changes to UMD spaces, 
We will refer to the book of Hyt\"onen, van Nerven, Vervaar and Weis \cite{HNVW} for the 
definitions and results about UMD spaces and Fourier multipliers, as well as to the paper of 
Zimmermann \cite{Zi}.

Let us recall some important features of UMD spaces. UMD spaces are reflexive 
\cite[p. 306]{HNVW}, with non-trivial type and finite cotype \cite[p. 313]{HNVW}, 
but the converse is not true (there exist reflexive Banach spaces with non-trivial type 
and finite cotype which are not UMD \cite[p. 311]{HNVW}). Moreover, a UMD space has an
equivalent uniformly convex norm (via super-reflexivity \cite[pp. 308 and 363]{HNVW}),
but the converse is false \cite[p. 354]{HNVW}.
The class of UMD spaces contains all $L^p$-spaces with $1<p<\infty$, and if $X$ is UMD, 
so is $L^p(X)$. Finally, let us mention that $X$ is UMD if and only if $X^*$ is \cite[p. 292]{HNVW}

%The Bourgain-Burkholder theorem \cite[p. 374]{HNVW} gives an analytic characterization
%of UMD spaces, by convergence and boundedness of the Hilbert transform in $L^p(\mathbb R,X)$.

\medskip

The  class of UMD spaces is the right one to work with Fourier Multipliers. 
In our context, the relevance of UMD spaces is that those spaces are precisely the ones for 
which the Riesz property (see below) holds, see for instance Theorem 5.2.10 page 398 of \cite{HNVW}. In particular the Marcinkiewicz theorem cannot 
hold on non-UMD spaces.

\medskip

{\bf Definition.} We say that $(a_n)_{n\in\Z}\in \C^\Z$ is an {\it $L^p(\T,X)$-Fourier multiplier}
if there exists $C_p>0$,  such that whenever $(c_n)_{n\in \Z} \in X^\Z$ and the series
$\sum_{n\in \Z}\gamma^n c_n$ converges in $L^p(\T,X)$, we have convergence of
$\sum_{n\in \Z}a_n \gamma^n c_n$, and
\begin{equation}\label{def-mult}
\int_\T \|\sum_{n\in \Z}a_n \gamma^n c_n\|^p \, d\gamma \le 
C_p^p \int_\T \|\sum_{n\in \Z} \gamma^n c_n\|^p\, d\gamma\, .
\end{equation}
Then, we denote by $\|(a_n)_{n\in \Z}\|_{\M_p(X)}$ the best constant $C_p$ for which \eqref{def-mult} holds.

\smallskip

\medskip

Set $I_0=\{0\}$ and for every $n \in \N$,  put $I_n=\{2^n,\ldots , 2^{n+1}-1\}$ and 
$I_{-n}=\{1-2^{n+1},\ldots , -2^n\}$. Given a sequence ${\bf a}=(a_n)_{n\in \Z}$ 
of complex numbers and an interval  $I=[\alpha,\beta]$ of integers, we define the 
{\it variation of ${\bf a}$ on $I$} by $V({\bf a},I):= \sum_{k=\alpha}^{\beta-1}|a_{k+1}-a_k|$. 
We define {\it the dyadic variation of ${\bf a}$}  by 
$V_d({\bf a}):=\sup_{n\in \Z}|a_n|+\sup_{n\in \Z}V({\bf a},I_n)$ and we say that ${\bf a}$ has bounded dyadic variation if $V_d({\bf a})<\infty$.

\medskip

In the following multiplier theorems, $X$ is (necessarily) a UMD space.  

\medskip

Zimmermann  \cite{Zi}  proved that any sequence with bounded dyadic variation is an 
$L^p(\T,X)$-Fourier multiplier, thus extending the Marcinkiewicz theorem 
which states the same result with $X=\C$. Moreover, there exists $C_p(X)>0$  such that $ \|{\bf a}\|_{\M_p(X)}\le C_p(X) V_d({\bf a})$.

In particular, for any 
interval  $I\subset \Z$, $(\delta_n(I))_{n\in \Z}$ is an $L^p(\T,X)$-Fourier multiplier and the norm $\|(\delta_n(I))_{n\in \Z}\|_{\M_p(X)}$ is bounded independently of $I$. We will call this result the {\it Riesz theorem}.

Moreover, any bounded monotone sequence of real numbers is an $L^p(\T,X)$-Fourier multiplier. 
We will call that result the {\it Stechkin theorem}.

Finally, any sequence with values in $\{-1,1\}$ that is constant on each dyadic interval is 
an $L^p(\T,X)$-Fourier multiplier. We will call that result the {\it Littlewood-Paley theorem.}

We are now in position to prove the result of this section. 

\begin{theorem} \label{kreiss-umd}
Let $X$ be a UMD Banach space. Let $q$ and $q^*$ be the (finite) cotypes of $X$ and $X^*$ 
respectively, and put $s=\min(q,q^*)$. Let $T$ be a Kreiss bounded operator on $X$. 
Then $\|T^n\|=O(n/(\log n)^{1/s})$; in particular $\|T^n\|=o(n)$.
\end{theorem}
\begin{proof} Assume the theorem is proved when $s=q$. When $s=q^*$,  we apply the result
to $T^*$ on $X^*$, noting that $T^*$ is also Kreiss bounded, and  by reflexivity $X^{**} = X$;
we obtain that $\|T^{*n}\|=O(n/(\log n)^{1/q^*})$, and use $\|T^n\|=\|T^{*n}\|$. 

Hence, we just have to prove the case where $s=q$. 
\smallskip

 By assumption, for every $r>1$ and every $\gamma \in \T$, we have for every $x\in X$ with 
$\|x\|=1$,
$$
\|\sum_{n\ge 0} \frac{\gamma^n T^nx}{r^{n+1}}\| = \|R(\bar\gamma r,T)\|  \le \frac{C}{(r-1)}\, .
$$

Let $p>1$. Let $N\in \N$ and take $r=1+1/N$. By the Riesz theorem (in $L^p(\T,X)$) there exists $C_p>0$ 
such that
\begin{gather*}
 \int_{\T} \|\sum_{n= 0}^{N-1} \frac{\gamma^nT^nx}{(1+1/N)^{n+1}}\|^p d\gamma \le  
C_p^p \int_{\T} \|\sum_{n\ge 0} \frac{\gamma^nT^nx}{(1+1/N)^{n+1}}\|^p \, d\gamma
\le (C_pC)^pN^p\, .
\end{gather*}

Define a sequence $(a_n)_{n\in \Z}$ as follows: $a_n=1+1/N$, if $n\le 0$, $a_n=(1+1/N)^{n+1}$ if $1\le n\le N$ and, $a_n=(1+1/N)^{N+1}$ if $n\ge N+1$. 
Then, $(a_n)_{n\in \Z}$ is bounded and monotone, hence, by the Stechkin theorem,  there exists $C_p>0$ such that 

\begin{equation*}
\int_{\TT} \Big\|\sum_{n=0}^{N-1} \gamma^n T^nx\Big\|^p  d\gamma  \le C_p^p \int_{\T} \|\sum_{n= 0}^{N-1} \frac{\gamma^nT^nx}{(1+1/N)^{n+1}}\|^p d\gamma \le C' N^p\, .
\end{equation*}

Using the Riesz theorem  again, for every $0\le M\le N-1$, we have 
\begin{equation}\label{est-UMD}
\int_{\T}\Big\|\sum_{n=M}^{N-1} \gamma^n T^nx\Big\|^p \,  d\gamma\le C'' N^p\, .
\end{equation}

Let $x^*\in X^*$ with $\|x^*\|=1$. Let $0\le P<Q\le N$ be integers. We have, writing $q':=q/(q-1)$ and using orthogonality, 

\begin{gather*}
(Q-P)|\langle x^*,T^Nx\rangle |= 
\Big|\int_{\T} \sum_{k=P}^{Q-1} \sum_{\ell=P}^{Q-1} 
\langle \bar \gamma^k T^{*k}x^*, \gamma^\ell T^{N-\ell}x\rangle d\gamma \Big|\\ 
\le \int_{\T} \Big\|\sum_{k=P}^{Q-1}\bar \gamma^kT^{*k}x^*\Big\|\,  \Big\|\sum_{\ell=P}^{Q-1} \gamma^{\ell} T^{N-\ell}x\Big\| \, d\gamma \\ \le \Big(  \int_{\T} \Big\|\sum_{k=P}^{Q-1} \bar \gamma^kT^{*k}x^*\Big\|^{q'}\, d\gamma\Big)^{1/q'}\, \Big(   \int_{|\T}  \Big\|\sum_{\ell=P}^{Q-1} \gamma^{ \ell} T^{N-\ell}x\Big\|^q \, d\gamma\Big)^{1/q}
\end{gather*}

Since, $T^*$ is also Kreiss bounded on $X^*$ (which is also UMD), \eqref{est-UMD} holds for $T^*$ with $p=q'$ and $N=Q$. Hence, taking the supremum over 
$\{\|x^*\|=1\}$, we see that 

\begin{equation}\label{est2-UMD}
 \int_{\T}  \Big\|\sum_{\ell=P}^{Q-1}\gamma^{\ell } T^{N-\ell}x\Big\|^q \, d\gamma\ge C\frac{(Q-P)^q}{Q^q}
\|T^Nx\|^q\, .
\end{equation}

\medskip

Let $N\in \N$. Let  $L:=\log (N/2)/\log 2$. Let $(\varepsilon_n)_{n\in \N}$ be Rademacher variables on some probability space $(\Omega,\F,\P)$. Using the Littlewood-Paley theorem, there exists $C_q>0$ such that 

\begin{gather*}
\int_{\T}\Big\|\sum_{n=N+1-2^L}^{N-1}\gamma^n T^n x\Big\|^q\, d\gamma  = \int_{\T}\Big\|\sum_{n=1}^{2^L-1}\gamma^n 
T^{N-n} x\Big\|^qd\gamma \\ \ge C_q \int_{\T}\Big\| \sum_{k=0}^{L-1} \varepsilon_k \sum_{\ell=2^k}^{2^{k+1}-1}\gamma^\ell T^{N-\ell}x\Big\|^q\, 
d\gamma  \, .
\end{gather*}

Using \eqref{LT-cotype}, we have
$$
\int_\Omega \Big\| \sum_{k=0}^{L-1} \varepsilon_k \sum_{\ell=2^k}^{2^{k+1}-1}\gamma^\ell T^{N-\ell}x\Big\|^qd\P \, \ge 
\sum_{k=0}^{L-1}\| \sum_{\ell=2^k}^{2^{k+1}-1}\gamma^\ell T^{N-\ell}x\Big\|^q
$$
Hence, 
$$
\int_{\T}\Big\|\sum_{n=N+1-2^L}^{N-1}\gamma^n T^n x\Big\|^q\, d\gamma \ge C_q\sum_{k=0}^{L-1}\int_\T \| \sum_{\ell=2^k}^{2^{k+1}-1}\gamma^\ell T^{N-\ell}x\Big\|^q\, d\gamma\, .
$$
In particular, using \eqref{est-UMD} and \eqref{est2-UMD}, we obtain that 
$$
C_qN^q \ge L\|T^Nx\|^q\, ,
$$
and the result follows.     
\end{proof}

%{\bf Remark.} When $X=L^p(\Omega,\mu)$ with $1<p<\infty$, we recover Proposition
%\ref{kreiss-Lp}, since then $s =2$. 

\begin{corollary} \label{kreiss-Lp}
Let $T$ be a Kreiss bounded operator on $L^p(\Omega,\mu)$, $1 < p< \infty$. Then 
$\|T^n\|=O(n/\sqrt{\log n})$.
\end{corollary}
\begin{proof} For $1<p< \infty$, \cite[p. 154]{AK} yields $s=2$.
\end{proof}

\begin{remark}The corollary extends the result proved for Hilbert spaces in 
\cite[Theorem 4.1]{CCEL} and in \cite{BM}.
\end{remark}

\begin{corollary}
Let $T$ be a uniformly Kreiss bounded operator on a UMD space. Then $\gamma T$
is mean ergodic for every $\gamma \in \mathbb T$.
\end{corollary}
\begin{proof} 
 By uniform Kreiss boundedness $\gamma T$ is Ces\`aro bounded (see 
 \cite{MSZ}),
and $\frac1n \|(\gamma T)^n\| = \frac1n \|T^n\| \to 0$ by Theorem \ref{kreiss-umd}.
\end{proof}
\bigskip

If we strengthen the Kreiss boundedness to strong Ces\`aro boundedness, we may drop the 
assumption that $X$ be UMD,  assuming only finite cotype for $X$ or $X^*$. 
Recall that there exist Banach spaces with finite cotype that are not UMD and even not reflexive; 
for instance, any $L^1$ space has cotype 2 \cite[p. 154]{AK}.

\begin{proposition}\label{prop-acb}
Let $T$ be a strongly Ces\`aro bounded operator on a Banach space $X$.  Let $q$ and $q^*$ be the cotypes of $X$ and $X^*$ 
respectively, and put $s=\min(q,q^*)$. Then $\|T^n\|=O(n/(\log n)^{1/s})$; in particular, if $s$ is finite, $\|T^n\|=o(n)$.
\end{proposition}
\begin{remark} 1. Of course the proposition is relevant only  if $s<\infty$. A positive 
Ces\`aro bounded operator on a Banach lattice is strongly Ces\`aro bounded 
\cite[Proposition 5.13]{CCEL}. Our result seems to be also new for positive operators. 

2. Assume that $X$ has type $p>1$, hence $X$ has cotype $p/(p-1)$. Then, as already mentionned 
$X^*$ has non trivial type, say $p^*>1$ and $X^{**}$ has cotype 
$p^*/(p^*-1)$. Now, if $T$ is strongly Ces\`aro bounded on $X$ so is 
$T^*$ on $X^*$ and the proposition gives $\|T^n\|=\|(T^*)^n\|=
O(n/(\log n)^{1/s})$ 
with $s:=\min(p/(p-1),p^*/(p^*-1)$.
\end{remark}

\begin{proof}
It follows from \eqref{SCB2} that $T$ is strongly Ces\`aro bouded if and only if $T^*$ is. 
Hence, as in the previous proof we may and DO assume that $s=q$. 
Let $(\varepsilon_n)_{n\in \N}$ be Rademacher variables.  By \eqref{SCB2} and the contraction principle, for every $p\in [1,\infty)$, there exists $C_p>0$ such that, for every $0\le M\le N-1$ and every $x\in X$, we have 
\begin{equation}\label{first-est}
\E (\|\sum_{n=M}^{N-1} \varepsilon_n T^n x\|^p) \le C_p N^p\|x\|^p\, .
\end{equation}
We have a similar estimate for $T^*$.

Assume that $s=q<\infty$ otherwise there is nothing to prove and set $q':=q/(q-1)$.  Notice that   for every $(x,x^*)\in X\times X^*$ and every 
$0 \le P<Q\le N$, we have 
$$
(Q-P)|\langle x^*, T^Nx\rangle |=\Big| \E\, \Big(\sum_{k=P}^{Q-1}\sum_{\ell =P}^{Q-1} \langle \varepsilon_kT^{*k}, \varepsilon_\ell T^{N-\ell}\rangle \Big)\Big|\, .
$$ 
Proceeding as in the previous proof, in particular taking supremum over $\{x^*\, : \, \|x^*\|=1\}$, we infer that for every $0\le P<Q\le N$ and every $x\in X$, 
\begin{equation}\label{second-est}
\E \Big(\Big\| \sum_{\ell=P}^{Q-1} \varepsilon_\ell T^{N-\ell}x\Big\|^q\Big)\ge C\frac{(Q-P)^q}{Q^q} \|T^Nx\|^q\, .
\end{equation}

\medskip

 Setting $L:= \log(N/2)\log 2$ and using \eqref{LT-cotype}, we infer that 
$$
\E\Big(\Big\|\sum_{n=N+1-2^L}^{N-1} \varepsilon_n T^n x\Big\|^q\Big) \ge C_q \sum_{k=0}^{L-1}\E\Big(\Big\|\sum_{\ell =2^k}^{2^{k+1}-1}
\varepsilon_\ell T^{N-\ell}\Big\|^q\Big)\, .
$$
Then the result follows by applying \eqref{first-est} with $p=q$ and $M=N+1-2^L$, combined 
with \eqref{second-est} (as in the previous proof).
\end{proof}

\begin{corollary}
Let $T$ be a strongly Ces\`aro bounded operator on $L^1(\Omega,\mu)$. Then 
$\|T^n\|=O(n/\sqrt{\log n})$.
\end{corollary}

\begin{remark} 1. The corollary improves the case $p=1$ of Corollary \ref{cor-Lp}, which yields 
the same result under the stronger assumption of absolute Ces\`aro boundedness. However, 
under this stronger assumption \cite[Proposition 3.1]{CCEL}  yields the better estimate 
$\|T^n\|=O(n^{1-\epsilon})$.

2. The corollary shows that the operator of Kosek \cite{Ko} is not strongly Ces\`aro bounded.
\end{remark}

\medskip

\begin{corollary}
Let $T$ be a positive Ces\`aro bounded operator on a Banach lattice with finite cotype $q$. Then, 
$\|T^n\|=O(n/( \log n)^{1/q})$. If $X=L^p(\Omega,\mu)$, with $1\le p<\infty$, then $\|T^n \|=O(n/\sqrt{\log n})$.
\end{corollary}
\noindent {\bf Proof.} By \cite[Prop. 5.13]{CCEL}, $T$ is strongly Ces\`aro bounded and we apply Proposition \ref{prop-acb}. When $X=L^p$ and $p\le 2$ we use the fact that $X$ has cotype $2$. When $p>2$, we apply the previous case to $T^*$, which is positive and Ces\`aro bounded.  \hfill $\square$

\begin{remark} As far as we know the corollary is new. The only result we are aware of in this direction is due to Emilion \cite{Emilion} and says that 
a positive Ces\`aro bounded operator on a reflexive Banach lattice $X$  satisfies $\|T^nx\|=o(n)$ for every $x\in X$. 
\end{remark}

\begin{corollary}
Let $T$ be a strongly Ces\`aro bounded operator on a reflexive Banach space $X$ such that $X$ or $X^*$ has  finite cotype. 
Then $\gamma T$ is mean ergodic for every $\gamma \in \mathbb T$.
\end{corollary}
\begin{remark}
 If $X$ is uniformly convexifiable norm, then it is reflexive with finite cotype and the corollary applies to any strongly Ces\`aro bounded operator on $X$.
 \end{remark}
%\begin{proof}
%Since $T$ is strongly Ces\`aro bounded, by \eqref{scb2}, $\gamma T$ is 
%mean ergodic for every complex number $\gamma$ with $|\gamma|=1$. By assumption $X$ or $X^*$ have finite c
%\end{proof}

\medskip

In view of the above results the following question seems natural.

{\bf Question.} Is every strongly Ces\`aro bounded operator on a 
reflexive Banach space mean ergodic ? 

\medskip

Notice that we even do not know whether a strongly Ces\`aro bounded operator on a reflexive Banach space is \emph{weakly} mean ergodic. 

\medskip

\noindent {\bf Acknowledgement.} This work has been partly motivated by a question raised by Markus Haase, after a talk of Vladimir 
M\"uller,  at  the 2019 
workshop of  the internet seminar on ergodic theorems organized  in Wuppertal. I would like to thank him here. I am also thankful to Guy Cohen, Tanja Eisner 
 and Michael Lin for  a careful reading of a preliminary version of the paper. Finally, I would like to thank the anonymous referee for his/her 
 extremely careful reading and for all his/her suggestions that improved the presentaion of the paper.

\end{document}